\documentclass[12pt]{amsart}

\usepackage{amsmath}
\usepackage{amsxtra}
\usepackage{amscd}
\usepackage{amsthm}
\usepackage{amsfonts}
\usepackage{amssymb}
\usepackage{eucal}

\newtheorem{thm}{Theorem}[section]
\newtheorem{cor}[thm]{Corollary}
\newtheorem{lem}[thm]{Lemma}
\newtheorem{prop}[thm]{Proposition}

\theoremstyle{definition}

\theoremstyle{remark}
\newtheorem{rem}[thm]{Remark}

\numberwithin{equation}{section}

\let\alb\allowbreak
\def\lsym#1{#1\alb\ldots\relax#1\alb} \def\lc{\lsym,}

\let\geq\geqslant

\let\leq\leqslant

\newcommand{\nc}{\newcommand}
\nc{\mc}{\mathcal}
\nc{\on}{\operatorname}
\nc{\Z}{{\mathbb Z}}
\nc{\C}{{\mathbb C}}
\nc{\pa}{\partial}
\nc{\arr}{\rightarrow}
\nc{\larr}{\longrightarrow}
\nc{\al}{\alpha}
\nc{\la}{\lambda}
\nc{\ep}{\epsilon}
\nc{\g}{{\mathfrak g}}
\nc{\h}{{\mathfrak h}}
\nc{\n}{{\mathfrak n}}
\nc{\Ga}{\Gamma}
\nc{\La}{\Lambda}
\nc{\si}{\sigma}
\nc{\bi}{\bibitem}
\nc{\om}{\omega}
\nc{\Res}{\on{Res}}
\nc{\ga}{\gamma}
\nc{\tr}{\on{tr}}
\nc{\bla}{\bold\lambda}

\nc{\gln}{{\mathfrak g\mathfrak l}_N}
\nc{\sln}{{\mathfrak s\mathfrak l}_N}
\nc{\V}{{\mc V}}

\newcommand{\beq}{\begin{equation}}
\newcommand{\eeq}{\end{equation}}
\newcommand{\bean}{\begin{eqnarray}}
\newcommand{\be}{\begin{equation*}}
\newcommand{\ee}{\end{equation*}}
\newcommand{\eean}{\end{eqnarray}}
\newcommand{\bea}{\begin{eqnarray*}}
\newcommand{\eea}{\end{eqnarray*}}
\newcommand{\bs}{\boldsymbol}
\newcommand{\Ref}[1]{{$($\ref{#1}$)$}}

\begin{document}

\title[Bethe algebra, Calogero-Moser space and Cherednik algebra]
{Bethe algebra of Gaudin model, Calogero-Moser space and Cherednik algebra}

\author[E. Mukhin, V. Tarasov, and A. Varchenko]
{E. Mukhin$\>^*$, V. Tarasov$\>^\star$, \and A. Varchenko$\>^\diamond$}

{\let\thefootnote\relax
\footnotetext{\vskip-.8pt\noindent
$^*$\,Supported in part by NSF grant DMS-0601005\\
{}$\quad$ $^\star$\,Supported in part by RFFI grant 08-01-00638\\
$^\diamond$\,Supported in part bynk NSF grant DMS-0555327}}

\maketitle

\begin{abstract}
We identify the Bethe algebra of the Gaudin model associated to
$\gln$ acting on a suitable representation with the center of the
rational Cherednik algebra and with the algebra of regular functions
on the Calogero-Moser space.
\end{abstract}

\maketitle

\section{Introduction}
The Bethe algebra of the Gaudin model associated to $\gln$ is a
remarkable commutative subalgebra of the universal enveloping algebra
of the current algebra of $\gln$. It is also known by the names of the
algebra of higher Gaudin Hamiltonians, see \cite{FFR}, or the algebra of
higher transfer matrices, see \cite{transfer}, or the quantum shift of argument
subalgebra, see \cite{FFRb}.

The Bethe algebra acts on a subspace $M$ of a given $\gln$-weight of
any $\gln[t]$-module producing a commutative family of linear
operators $\mc B(M)\in\on{End} M$. The main problem of the Gaudin
model is to describe common eigenvectors and eigenvalues of this
family.

It often turns out that the Bethe algebra $\mc B(M)$ can be naturally
identified with the algebra of regular functions $\mc O_{\mc X}$ on an
affine variety $\mc X$, and $M$ becomes the regular representation of
$\mc O_{\mc X}$. Then the common eigenvectors of $\mc B(M)$ are in a
bijective correspondence with the points of $\mc X$, the joint
spectrum of $\mc B(M)$ is simple and $\mc B(M)$ is a maximal
commutative subalgebra of $\on{End} M$. One can also hope to get new
information about the variety $\mc X$ by studying the algebra $\mc
B(M)$.

Such an idea was realized in \cite{transversal}, \cite{exp}, where $M$
is a subspace of a given weight of an arbitrary finite-dimensional
irreducible $\gln[t]$-module. Then $\mc B(M)$ is a finite-dimensional
algebra and $\mc X$ is the scheme-theoretic intersection of a suitable
Schubert varieties in a Grassmannian of $N$ planes. Besides the new
information on the spectrum of the Gaudin model, on the
algebro-geometric side this study gave a proof of the B. and M.
Shapiro conjecture, see \cite{shapiro}, a proof of the transversality
conjecture discussed in \cite{S}, and an effective proof of the reality of
the Schubert calculus.

In this paper, we consider the action of $\gln[t]$ on polynomials in
several variables with values in a tensor product of vector
representations of $\gln$, which is defined via the evaluation map,
and the subspace $M$ of polynomials with values in the zero
$sl_N$-weight subspace of the tensor product. Thus, in contrast to
\cite{transversal}, \cite{exp}, we keep the evaluation parameters and
the parameters in the Bethe algebra formal variables, which makes $M$
an infinite-dimensional module. In this case, we show that the
corresponding affine variety $\mc X$ is the Calogero-Moser space. The
Calogero-Moser is a celebrated affine non-singular variety, which
appears in many areas of mathematics, see \cite{KKS}, \cite{Wn},
\cite{EG}, \cite{CBH}.

We prove two main theorems. First, we show that $\mc B(M)$ is naturally
isomorphic to the center of the rational Cherednik algebra of type
$A$, see Theorem \ref{Z=B}. Second, we show that $\mc B(M)$ is
naturally isomorphic to the algebra $\mc O_{\mc X}$ of regular
functions on the Calogero-Moser space, see Theorem \ref{C=B}.

The Bethe algebra $\mc B(M)$ is generated by coefficients of a row
determinant of some matrix, see \cite{CT}, \cite{transfer}. We show that
the center of the Cherednik algebra is generated by coefficients of an
explicit determinant-like formula and that the algebra $\mc O_{\mc X}$
of regular functions on the Calogero-Moser space is generated by the
coefficients of the polynomial version of the Wilson $\Psi$ function.
The isomorphisms in Theorems \ref{Z=B} and \ref{C=B} just send the
corresponding coefficients to each other. The proof of Theorem
\ref{Z=B} is a simple algebraic argument. The proof of Theorem
\ref{C=B} follows the logic of \cite{transversal} and it is more
involved. In particular, we use the machinery of the Bethe ansatz and
the Wilson correspondence of the Calogero-Moser space to the adelic
Grassmannian.

It is proved in \cite{EG} that the center of the Cherednik algebra of
type $A$ is isomorphic to the algebra of regular functions on the
Calogero-Moser space. We recover this result.

The paper is organized as follows. We start with identifying the Bethe
algebra with the center of the Cherednik algebra in Section \ref{sec
Z=B}. We discuss this result in Section \ref{sec cor}. In
particular, we give an explanation of Theorem \ref{Z=B} using a
general construction of a commutative subalgebra from the center of an
algebra, see Sections \ref{comm sec} and \ref{another sec}. In
Section \ref{sec C=B} we describe the map between the Bethe algebra
and the algebra of regular functions on the Calogero-Moser space. We
give the proof that this map is a well-defined isomorphism of algebras
in Section \ref{proof C=B}. Some corollaries of this isomorphism are
given in Section \ref{sec cor 2}. In particular, Section \ref{sec
bij} describes the bijections between eigenvectors of the Bethe
algebra $\mc B(M)$ and three sets which are known to be equivalent:
points of the Calogero-Moser space, the points of the
adelic Grassmannian and the set of irreducible
representations of the Cherednik algebra.
These bijections follow from Theorems \ref{Z=B}
and \ref{C=B}.

We thank E. Vasserot and P. Etingof for useful discussions.

\section{Bethe algebra of the Gaudin model and the center of the
Cherednik algebra}\label{sec Z=B}
\subsection{Multi-symmetric polynomials}
Let $\C[\bs z,\bs \la]=\C[z_1\lc z_N,\la_1\lc\la_N]$ be the
algebra of polynomials in commuting variables.

Let $S_N$ be the group of permutations of $N$ elements. We
often consider actions of the group $S_N$ which permute indices in the
groups of $N$ variables, in such cases we will indicate the affected
group of variables by the upper indices. For example, $S_N^z$ permutes
the variables $z_1\lc z_N$ and no other variables, $S_N^{z,\la}$
permutes the variables $z_1\lc z_N$ and the variables
$\la_1\lc\la_N$, etc. We use the same notation for the elements of
$S_N$. For example, for $\si,\tau\in S_N$,

\be
(\si^z\tau^\la) \bigl(p(z_1\lc z_N,\la_1\lc\la_N))=
p(z_{\si(1)}\lc z_{\si(N)}, \la_{\tau(1)}\lc\la_{\tau(N)}).
\ee
We also have $\si^z\tau^\la=\tau^\la\si^z$ and
$\si^z\si^\la=\si^{z,\la}$.

Let $P_N=\C[\bs z,\bs \la]^{S_N^{z,\la}}\subset \C[\bs z,\bs \la]$ be
the algebra of polynomials invariant with respect to simultaneous
permutations of $z_1,\dots, z_N$ and $\la_1,\dots,\la_N$.
We call $P_N$ the {\it algebra of
multi-symmetric polynomials}. The algebra $P_N$ is also known by the
names of MacMahon polynomials, vector symmetric polynomials,
diagonally symmetric polynomials, etc, it is well-studied, see for
example \cite{W}.

Consider the algebras $\C[\bs z]^{S_N^z}$ and $\C[\bs
\la]^{S_N^{\la}}$ of symmetric polynomials in $\bs z$ and $\bs \la$
respectively. We have an obvious inclusion $\C[\bs z]^{S_N^z}\otimes
\C[\bs \la]^{S_N^\la} \to P_N$ given by the multiplication map. The
following lemma is a standard fact.

\begin{lem}
The algebra $P_N$ is a free $\C[\bs z]^{S_N^z}\otimes \C[\bs
\la]^{S_N^\la}$-module of rank $N!$.
\qed
\end{lem}

Consider the wreath product $\C[\bs z,\bs \la]\ltimes \C S_N^{z,\la}$.
We write the elements of $\C[\bs z,\bs \la]\ltimes \C S_N^{z,\la}$ in
the form $\sum_{\si\in S_N} p_\si(\bs z,\bs \la)\si$, where $p_\si(\bs
z,\bs\la)\in\C[\bs z,\bs\la]$. Such an element is zero if and only if
all $p_\si=0$. The algebra $\C[\bs z,\bs \la]$ is embedded in $\C[\bs
z,\bs \la]\ltimes \C S_N^{z,\la}$ by the map $p(\bs z,\bs\la)\mapsto
p(\bs z,\bs\la)\ id$.

The following is another standard fact.
\begin{lem}\label{smash center}
The algebra $P_N \cdot id$ is the center of $\C[\bs z,\bs
\la]\ltimes \C S_N^{z,\la}$. \qed
\end{lem}

Define the {\it universal multi-symmetric polynomial}
\be
\mc P^P=\prod_{i=1}^N((u-z_i)(v-\la_i)-1).
\ee
It is a polynomial in variables $u,v$ with coefficients in $P_N$.
Write
\be
\mc P^P=\sum_{i,j=0}^N p_{ij} u^{N-i}v^{N-j},
\qquad p_{ij}=p_{ij}(\bs z,\bs\la)\in P_N.
\ee
We have $p_{00}=1$.
\begin{lem}\label{multisym gen}
The polynomials $p_{ij}$, $i,j=0,1\lc N$, generate the algebra
$P_N$.
\end{lem}
\begin{proof}
Let $Z=\on{diag}(z_1\lc z_N)$, $\La=\on{diag}(\la_1\lc\la_N)$
be the diagonal $N\times N$ matrices. Then $\mc
P^P=\det((u-Z)(v-\La)-1)$ and
\be
\det(u-Z)=u^N+\sum_{i=1}^Np_{i0}u^{N-i}, \qquad
\det(v-\La)=v^N+\sum_{j=1}^Np_{0j}v^{N-j}.
\ee
Therefore, the coefficients of the series
$\log\det(1-(u-Z)^{-1}(v-\La)^{-1})$ in $u^{-1}, v^{-1}$
are polynomials in $p_{ij}$. We have
\begin{align*}
& \log \det (1-(u-Z)^{-1}(v-\La)^{-1})\,=\,
\tr(\log(1-(u-Z)^{-1}(v-\La)^{-1}))\,={}\\
&\,{}-\sum_{r=1}^\infty\frac{1}{r}\tr\left((\sum_{i=0}^\infty Z^iu^{-i-1})
(\sum_{j=0}^\infty \La^jv^{-j-1})\right)^r=\,
-\sum_{i,j,k,l=0}^\infty c_{ij}^{kl}\ \tr(\La^kZ^l)u^{-i-1}v^{-j-1}\,,
\end{align*}
where $c_{ij}^{kl}$ are rational numbers.
In the last equality we used $Z\La=\La Z$.
Then $c_{ij}^{ij}=1$ and $c_{ij}^{kl}=0$ if $k>i$ or $l>j$.
Therefore, by triangularity, $\tr(\La^kZ^l)$ are polynomials in $p_{ij}$.

It is well-known that the power sums multi-symmetric polynomials
$\tr(\La^kZ^l)=\sum_{i=1}^N \la_i^kz_i^l$ generate $P_N$, see
\cite{W}. The lemma follows.
\end{proof}

\subsection{The Bethe algebra}
Let $\gln$ denote the complex Lie algebra of all $N\times N$ matrices
and $U\gln$ its universal enveloping algebra. The algebra $U\gln$ is
generated by the elements $e_{ij}$, $i,j=1\lc N,$ satisfying the
relations $[e_{ij},e_{sk}]=\delta_{js}e_{ik}-\delta_{ik}e_{sj}$.

Let $\gln[t]$ denote the current algebra of $\gln$ and $U(\gln[t])$
its universal enveloping algebra. The algebra $U(\gln[t])$ is
generated by elements $e_{ij}\otimes t^r$, $i,j=1\lc N$,
$r\in\Z_{\geq 0}$, satisfying the relations
$[e_{ij}\otimes t^r,e_{sk}\otimes t^p]=\delta_{js}e_{ik}\otimes
t^{r+p}-\delta_{ik}e_{sj}\otimes t^{r+p}$.

It is convenient to collect elements of $\gln[t]$ in generating series
of the variable $u$. Namely, for $g\in\gln$, set
\be
g(u)=\sum_{s=0}^\infty (g\otimes t^s)u^{-s-1}.
\ee

Let $\bs\la=(\la_1\lc\la_N)$ be a sequence of formal commuting
variables. Denote the algebra of polynomials in variables $\la_1\lc\la_N$
with values in $U(\gln[t])$ by $U(\gln[t])[\bs\la]$.

We define the {\it row determinant} of an $N\times N$ matrix $A$
with entries $a_{ij}$ in a possibly non-commutative algebra to be
\beq\label{rdet}
\on{rdet} A\,=\,
\sum_{\sigma\in S_N} a_{1\sigma(1)}a_{2\sigma(2)}\dots a_{N\sigma(N)}.
\eeq

Denote the operator of the formal differentiation
with respect to the variable $u$ by $\partial$.
Define the {\it universal operator}
$\mc D^{\mc B}$ by
\be
\mc D^{\mc B}=\on{rdet}
\left( \begin{matrix}
\partial-\la_1-e_{11}(u) & -e_{21}(u)& \dots & -e_{N1}(u)\\
-e_{12}(u) &\partial-\la_2 -e_{22}(u)& \dots & -e_{N2}(u)\\
\dots & \dots &\dots &\dots \\
-e_{1N}(u) & -e_{2N}(u)& \dots & \partial-\la_N-e_{NN}(u)
\end{matrix}\right).
\ee

The universal operator $\mc D^{\mc B}$ is a differential operator in
the variable $u$ whose coefficients are formal power series in
$u^{-1}$ with coefficients in $U(\gln[t])[\bs\la]$. Write

\be
\mc D^{\mc B}= \partial^N+\sum_{i=1}^N B_i(u)\partial^{N-i}, \qquad
B_i(u)=\sum_{j=0}^\infty B_{ij}u^{-j},\qquad B_{ij}\in U(\gln[t])[\bs\la].
\ee


We call the unital subalgebra of $U(\gln[t])[\bs\la]$ generated by
$B_{ij}$, $i=1\lc N$, $j\in\Z_{\geq 0}$, the {\it Bethe algebra}
and denote it by $\mc B_N$.

\begin{lem}[\cite{CT},\cite{transfer}]\label{bethe}
The algebra $\mc B_N$ is commutative. The algebra $\mc B_N$
commutes with $e_{ii}$ and multiplications by $\la_i$, $i=1\lc N$.
\qed
\end{lem}

As a subalgebra of $U(\gln[t])[\bs\la]$, $\mc B_N$ acts on any
$U(\gln[t])[\bs\la]$-module $M$. Since $\mc B_N$ commutes with
$e_{ii}$, it preserves the $\gln$-weight decomposition of the module
$M$.

\subsection{The Cherednik algebra}
Denote by $H_N$ the {\it rational Cherednik algebra associated with
the symmetric group $S_N$}. The algebra $H_N$ is the unital complex
algebra with generators $x_i,y_i,s_{jk}$, where $i,j,k=1\lc N$,
$j\neq k$, and relations
\begin{gather*}
s_{ij}=s_{ji}, \quad s^2_{ij}=1,
\quad s_{ij} s_{jk} =s_{ik}s_{ij}, \quad s_{ij}s_{kl}=s_{kl}s_{ij}, \\
[x_i,x_j] = [ y_i,y_j ]=0,\\
s_{ij}x_i=x_js_{ij},\quad s_{ij}y_i=y_js_{ij},
\quad [s_{ij},x_k]=[s_{ij},y_k]=0, \\
\ [x_i,y_j]=s_{ij},\quad [x_i,y_i]=-\sum_{a,\ a\neq i}s_{ia},
\end{gather*}
where in each relation all the indices are distinct elements of
$\{1\lc N\}$. The rational Cherednik algebra $H_N$ is a two step
degeneration of the double affine Hecke algebra, see \cite{Ch}.

We employ the notation $\C[\bs x]=\C[x_1\lc x_N]$ and $\C[\bs
y]=\C[y_1\lc y_N]$. The algebra $H_N$ is a deformation of the
wreath product of the algebra of polynomials in commuting variables
$\C[\bs x]\otimes \C[\bs y]$ and of the group algebra $\C[S_N^{x,y}]$
generated by transpositions $s_{ij}$. In particular, we have a linear
isomorphism given by the multiplication map:
\begin{gather*}
\C[\bs x]\otimes \C[S_N]\otimes\C[\bs y]\,\to\, H_N,\\
q\otimes\sigma\otimes p \,\mapsto\, q\ \sigma \ p.
\end{gather*}
We call this isomorphism the {\it normal ordering map} and denote it by $::$.

For example, $:y_1x_2y_3x_4:=x_2x_4y_1y_3$,
$:x_1y_1^2x_2:=x_1x_2y_1^2$, etc.
Note that we omit the tensor signs in writing the elements of
$\C[\bs x]\otimes \C[S_N]\otimes\C[\bs y]$.

Denote by $\mc Z_N\subset H_N$ the center of $H_N$.

Define the {\it universal central polynomial}
\be \mc
P^{\mc Z}=(-1)^N\sum_{\sigma\in S_N}
:\prod_{i,\ \sigma(i)=i} (1-(v-x_i)(u-y_i)):\
(-1)^\sigma\sigma.
\ee
The polynomial $\mc P^{\mc Z}$ is a polynomial in
$u$ and $v$ with coefficients in $H_N$. Write
\be
\mc P^{\mc Z}=\sum_{i,j=0}^N c_{ij}\ v^{N-i}u^{N-j},\qquad c_{ij}\in H_N.
\ee

\begin{thm}\label{central thm}
The elements $c_{ij}$, $i,j=0\lc N$, generate the center
$\mc Z_N\subset H_N$.
\end{thm}
\begin{proof}
First, we show that the elements $c_{ij}$ are central.
It is clear that $s_{ij}\mc P^{\mc Z}=\mc P^{\mc Z}s_{ij}$ for any $i,j$.
Hence, to show that $[x_i,\mc P^{\mc Z}]=[y_i,\mc P^{\mc Z}]=0$ for all
$i=1\lc N$, it is enough to check these equalities for $i=1$ only.

We begin with proving that $[y_1,\mc P^{\mc Z}]=0$.
Let
\be
a_i=(1-(v-x_i)(u-y_i)), \qquad A_\sigma=:\prod_{i,\ \sigma(i)=i}a_i:\,,
\ee
so that
\be
\mc P^{\mc Z}=(-1)^N\sum_{\sigma \in S_N}(-1)^\sigma A_\sigma \sigma\,.
\ee
Let
\be
[y_1,\mc P^{\mc Z}]=(-1)^N\sum_{\sigma\in S_N} (-1)^\sigma C_\sigma \sigma,
\ee
where $C_\sigma$ has the form
\be
C_\sigma=\sum_{i=1}^{k_{\sigma}} p_{i,\sigma}(\bs x)q_{i,\si}(\bs y).
\ee
Then we have
\be
C_\sigma=A_\sigma(y_1-y_{\sigma(1)})-
\sum_{i=2}^N \phi([y_1,A_{s_{1i}\sigma}] s_{1i}),
\ee
where $\phi$ is a linear map
\be
\phi: H_N\to H_N,\qquad
p(\bs x)q(\bs y)\si\mapsto \delta_{id,\si}p(\bs x)q(\bs y).
\ee

The expression $\phi([y_1,A_{s_{1i}\sigma}] s_{1i})$ equals zero unless
$\sigma(1)=i$ or $\sigma(i)=1$. Hence, $C_\sigma=0$ if $\sigma(1)=1$.

Assume that $\sigma(1)=k\ne 1$ and $\sigma(k)=1$. Then
\begin{align*}
C_\sigma\,=\,A_\sigma(y_1-y_k) &{}-\phi([y_1,A_{s_{1k}\sigma}] s_{1k})\,={}
\\[3pt]
{}=\,A_\sigma(y_1-y_k) &{}-(A_\sigma(u-y_k)-A_\sigma(u-y_1)-{}
\\
&{}-(v-x_1)A_\sigma(u-y_1)(u-y_k)+(v-x_1)A_\sigma(u-y_1)(u-y_k))\,=\,0.
\end{align*}

Assume that $\sigma(1)=k\ne 1$, $\sigma(l)=1$, and $k\ne l$. Then
\begin{align*}
C_\sigma\,=\,A_\sigma(y_1-y_k) &{}-\phi([y_1,A_{s_{1k}\sigma}] s_{1k})
-\phi([y_1,A_{s_{1l}\sigma}] s_{1l})\,={}
\\[3pt]
{}=\,A_\sigma(y_1-y_k) &{}-A_\sigma(u-y_k)+A_\sigma(u-y_1)\,=\,0.
\end{align*}

The proof of $[x_1,\mc P^{\mc Z}]=0$ is similar with the following
modification: we use
\be
\mc P^{\mc Z}=(-1)^N\sum_{\sigma \in S_N}(-1)^\sigma \sigma A_\sigma
\ee
and move elements $\sigma\in S_N$ to the left.

There is a filtration on $H_N$ given by $\deg x_i=\deg y_i=1$, $\deg s_{ij}=0$.
The associated graded ring is isomorphic to the wreath product
$\C[\bs z,\bs\la]\ltimes\C S_N^{z,\la}$, and all graded components are
finite-dimensional. The center of $H_N$ is projected into the center of
the wreath product, which is the algebra $P_N$, see Lemma~\ref{smash center}.
The elements $c_{ij}$ are projected to the elements $p_{ij}$, which are
generators of $P_N$ by Lemma \ref{multisym gen}. Hence, the elements $c_{ij}$
generate the center of $H_N$.
\end{proof}

A {\it central character} of $H_N$ is an algebra homomorphism
$\chi: \mc Z_N\to \C$. Central characters determine irreducible representations
of $H_N$, see Theorem~1.24 in \cite{EG}. Let
\beq\label{symm el}
e=\frac{1}{N!}\sum_{\sigma\in S_N}\sigma\in H_N
\eeq
be the symmetrizing element.

\begin{thm}[\cite{EG}]\label{irreps}
Any irreducible $H_N$-module has dimension $N!$ and is isomorphic to
the regular representation of $S_N$ as an $S_N$-module.
Irreducible $H_N$-modules are in a bijective correspondence with algebra
homomorphisms $\chi: \mc Z_N\to \C$.
The irreducible $H_N$-module corresponding to the central character $\chi$ is
given by $H_Ne\ \otimes_{\mc Z_N}\chi$.
\qed
\end{thm}
We denote the set of isomorphism classes of irreducible modules
of the Cherednik algebra $H_N$ by $R_N$.

\subsection{The space $\V_{\bs 1}$}
Let $V$ be the vector representation of $\gln$, $\dim V=N$. Let
$\ep_1\lc\ep_N$ be the standard basis of $V$,
$e_{ij}\ep_k=\delta_{jk}\ep_i$.

Let $\V$ be the space of polynomials in commuting variables
$z_1\lc z_N$ and $\la_1\lc\la_N$ with coefficients in
$V^{\otimes N}$:

\be
\V = V^{\otimes N} \otimes_{\C}\C[\bs z, \bs \la].
\ee
For $v\in V^{\otimes N}$ and
$p(\bs z,\bs\la)\in\C[\bs z, \bs \la]$, we write
$p(\bs z,\bs \la)\,v$ instead of $v\otimes p(\bs z,\bs\la)$.

Denote the subspace of $V^{\otimes N}$ of $\gln$-weight
$(1,1\lc 1)$ by $(V^{\otimes N})_{\bs 1}$:
\be
(V^{\otimes N})_{\bs 1}=\{v\in V^{\otimes N}\ |\ e_{ii}v=v,\ i=1\lc N\}.
\ee
We have $\dim (V^{\otimes N})_{\bs 1}=N!$.
A basis of $(V^{\otimes N})_{\bs 1}$ is given by vectors
\be
\ep_\si=\ep_{\si(1)}\otimes\dots\otimes\ep_{\si(N)}\,,\qquad\sigma\in S_N\,.
\ee
Let
$\V_{\bs 1}= (V^{\otimes N})_{\bs 1}\otimes_{\C}\C[\bs z, \bs \la]$.

Consider the space $\V$ as a $U(\gln[t])[\bs\la]$-module with
the series $g(u)$, \,$g\in\gln$, acting by 
\be
\label{action}
g(u)\,\bigl(p(\bs z,\bs\la)\,v_1\otimes\dots\otimes v_N)\,=\,
p(\bs z,\bs\la)\,\sum_{i=1}^N
\frac{v_1\otimes\dots\otimes gv_i\otimes\dots\otimes v_N}{u-z_i}\ .
\ee
and the algebra $\C[\bs\la]$ acting by multiplication operators. In
particular, the Bethe algebra $\mc B_N$ acts on $\V$. Since $\mc B_N$
commutes with $e_{ii}$, the Bethe algebra $\mc B_N$ also acts on the
space $\V_{\bs 1}$. We denote the image of $\mc B_N$ in
$\on{End}(\V_{\bs 1})$ by $\bar{\mc B}_N$.

We compute the action of the coefficients of the universal operator
$\mc D^{\mc B}$ in $\V_{\bs 1}$. Let $\bar B_{ij}\in\on{End}(\V_{\bs
1})$ be the images of the operators $B_{ij}$.

Define the {\it universal Bethe polynomial} $\mc P^{\bar {\mc B}}$
by the formula
\be
{\mc P}^{\bar{\mc B}}=w(u,\bs z)\left(v^N+\sum_{i=1}^N
\sum_{j=0}^\infty \bar B_{ij}u^{-j}v^{N-i}\right), \qquad
w(u,\bs z)=\prod_{i=1}^N(u-z_i).
\ee
The universal
Bethe polynomial is a polynomial in $u$ and $v$ with
coefficients in $\on{End}(\V_{\bs 1})$, see \cite{capelli}.
Write
\be
{\mc P}^{\bar{\mc B}}=
\sum_{i,j=0}^N\bar b_{ij}u^{N-i}v^{N-j},
\qquad\bar b_{ij}\in \on{End}(\V_{\bs 1}).
\ee

\begin{lem}
The algebra $\bar{\mc B}_N$ is generated by $\bar b_{ij}$,
$i,j=0,1\lc N$.
\end{lem}
\begin{proof}
We have
\be
w(u,\bs z)=\sum_{i=0}^N\bar b_{i0}u^{N-i}.
\ee
Therefore the coefficients of the power series
${\mc P}^{\bar{\mc B}}/w(u,\bs z)$ are in the algebra generated by
$\bar b_{ij}$.
\end{proof}

\begin{lem}\label{comp bethe}
For all $\tau\in S_N$, we have:
\be
{\mc P}^{\bar{\mc B}} \ep_\tau= (-1)^N
\sum_{\si\in {S_N}} (-1)^\si\prod_{i,\
\sigma(i)=i}(1-(u-z_{\tau^{-1}(i)})(v-\la_i))\ \ep_{\si\tau}.
\ee
Moreover,
the operators $\bar b_{ij}$ commute with multiplications by $\la_s$ and $z_s$.
\end{lem}

\begin{proof}
Following \cite{capelli}, to compute the polynomial
$(-1)^N {\mc P}^{\bar B}$ we have to
consider the determinant
\be
\on{rdet}
\left( \begin{matrix}
\sum_{i=1}^N \frac{e_{11}^{(i)}}{u-z_i}-(v-\la_1) & \sum_{i=1}^N
\frac{e_{21}^{(i)}}{u-z_i}& \dots &
\sum_{i=1}^N \frac{e_{N1}^{(i)}}{u-z_i}\\
\sum_{i=1}^N \frac{e_{12}^{(i)}}{u-z_i} &
\sum_{i=1}^N \frac{e_{22}^{(i)}}{u-z_i}-(v-\la_2)& \dots &
\sum_{i=1}^N \frac{e_{N2}^{(i)}}{u-z_i}\\
\dots & \dots &\dots &\dots \\
\sum_{i=1}^N \frac{e_{1N}^{(i)}}{u-z_i} &
\sum_{i=1}^N \frac{e_{2N}^{(i)}}{u-z_i}& \dots &
\sum_{i=1}^N \frac{e_{NN}^{(i)}}{u-z_i} - (v-\la_N)
\end{matrix}\right),
\ee
expand it, ignore all the terms with poles of order higher
than one, compute the action of $e_{jk}^{(i)}$,
and finally multiply by $w(u,\bs z)$. Here
$e_{jk}^{(i)}$ denote the operators $e_{jk}$
acting on the $i$-th factor in $V^{\otimes N}$.

The terms $A_\si$ in the expansion are labeled by permutations $\si\in S_N$,
see \Ref{rdet}, and
\be
A_\si=(e_{\si(1)1}(u)-\delta_{\si(1)1}(v-\la_1))
(e_{\si(2)2}(u)-\delta_{\si(2)2}(v-\la_2))\dots
(e_{\si(N)N}(u)-\delta_{\si(N)N}(v-\la_N)).
\ee
Note also that
\be
e_{ij}(u)\ep_\tau=\frac{e^{(\tau^{-1}(j))}_{ij}}{u-z_{\tau^{-1}(j)}}\ \ep_\tau.
\ee
Therefore
\be
A_\si \ep_\tau=\prod_{i,\
\sigma(i)=i}\left(\frac{1}{u-z_{\tau^{-1}(i)}}-
(v-\la_i)\right)\ \ep_{\si\tau}+\dots,
\ee
where the dots denote the terms with at least one pole in $u$
of order greater than $1$.

The lemma follows.
\end{proof}

\subsection{The Bethe algebra and the center of the Cherednik algebra}
We identify the space $\V_{\bs 1}$ with $H_N$ as follows. Let $\iota$
be the isomorphism of vector spaces given by
\be
\iota:\ \V_{\bs 1}\to H_N, \qquad p(\bs z)q(\bs \la) \ep_\tau
\mapsto q(\bs x)\tau p(\bs y),
\ee
for all $\tau\in S_N$ and all polynomials $p, q$.

The map $\iota$ identifies the action of the Bethe
algebra $\bar{\mc B}_N$ in $\V_{\bs 1}$ with the action of the
center $\mc Z_N$ in $H_N$. Namely, we have the following theorem.
\begin{thm}\label{Z=B}
We have
\be
\iota \ {\mc P}^{\bar{\mc B}}= \mc P^{\mc Z} \iota .
\ee
\end{thm}
\begin{proof}
It is sufficient to check the equality in the theorem on
elements of $\V_{\bs 1}$ of the form $p(\bs z)q(\bs \la)\ep_{\tau}$,
where $\tau\in S_N$ and $p,q$ are polynomials.

The left hand
side, $\iota \ {\mc P}^{\bar{\mc B}}(p(\bs z)q(\bs \la)\ep_{\tau})$,
is computed using Lemma 2.8.

For the right hand side, we have
\be
\mc P^{\mc Z} \iota (p(\bs z)q(\bs \la)\ep_{\tau})=\mc P^{\mc Z} q(\bs x)
\tau p(\bs y)=q(\bs x)\mc P^{\mc Z }\tau p(\bs y),
\ee
where in the second equality we used that $\mc P^{\mc Z}$ is
central by Theorem \ref{central thm}.

Furthermore, if $\si(i)=i$ then
\be
q(\bs x) (v-x_i)(u-y_i)\sigma \tau p(\bs y)=
q(\bs x)(v-\la_i)\si \tau (u-y_{\tau^{-1}(i)})p(\bs y).
\ee

The theorem follows.
\end{proof}

\begin{cor}
There exists a
unique isomorphism of algebras
$\tau_{BZ}:\ \bar{\mc B}_N\to \mc Z_N$ which maps
$\bar b_{ij}$ to $c_{ij}$.
\qed
\end{cor}

\section{Remarks on Theorem \ref{Z=B}}\label{sec cor}
\subsection{A construction of a commutative algebra}\label{comm sec}
We describe a useful way to construct commutative subalgebras
from the center of an algebra.

Let $A$ be an algebra and let $\mc Z_A\subset A$ be the center of $A$.
Assume that $A_{+}, A_{-}\subset A$ are subalgebras of $A$ such that
$A=A_{+}A_{-}$. By that we mean that the multiplication map
$A_+\otimes A_-\to A$ is an isomorphism of vector spaces.

Let $A_{+}^{op}$ be the algebra $A_{+}$ with the opposite
multiplication: $A_{+}^{op}=A_{+}$ as vector spaces and the
multiplication map $A_+^{op}\otimes A_+^{op}\to A_+^{op}$ sends
$a_{+}\otimes b_{+}$ to $b_{+}a_{+}$.

We have a unique isomorphism of vector spaces defined by
\be
\al: A\to A_{+}^{op}\otimes A_{-}, \qquad
a_{+}a_{-} \mapsto a_{+}\otimes a_{-},
\ee
for all $a_+\in A_+, a_-\in A_-$.

\begin{lem}\label{elementary}
The algebra $\al(\mc Z_{A})$ is a commutative subalgebra of
$A_{+}^{op}\otimes A_{-}$ isomorphic to $\mc Z_A$.
\end{lem}
\begin{proof}
Let $a=\sum_{i} a_{+}^{(i)}a_-^{(i)}$ and $b=\sum_{j}
b_{+}^{(j)}b_-^{(j)}$ be elements of the center $\mc Z_A$. Here
$a_\pm^{(i)},b_\pm^{(j)}\in A_\pm$. We have
\bea
\al^{-1}([\al(a),\al(b)])
=\sum_{ij}([b_+^{(j)},a_+^{{(i)}}]a_-^{(i)}b_-^{(j)}+
a_+^{(i)},b_+^{(j)}[a_-^{(i)}b_-^{(j)}])=\\
\sum_{ij}(-a_+^{(i)}[b_+^{(j)},a_-^{(i)}]b_-^{(j)}+
a_+^{(i)},b_+^{(j)}[a_-^{(i)}b_-^{(j)}])=0.
\eea
Here the first equality follows from the definitions, the second
from the centrality of $a$ and the third one from the centrality of $b$.
\end{proof}

\begin{rem}
The idea of this construction can be eventually traced back to Kostant-Adler
method in the theory of integrable systems, see \cite{K}, \cite{ReS}.
A similar idea
in a disguised form is involved in the factorization method, see \cite{RS} and
the construction of higher Gaudin Hamiltonians, see \cite{FFR}, \cite{FFRb}.
\end{rem}

\subsection{Another form of Theorem \ref{Z=B}}\label{another sec}
We interpret Theorem \ref{Z=B} as a coincidence of
two natural commutative subalgebras in the algebra
$\C[\bs x]\otimes (\C[\bs y]\ltimes \C[S_N^{y}])$.

Let $A=H_N$, $A_+=\C[\bs x]=A_+^{op}$ and $A_-=\C[\bs y]\ltimes
\C[S_N^y]$. Then by Lemma \ref{elementary} we have a commutative
subalgebra $\al(\mc Z_N)\subset \C[\bs x]\otimes(\C[\bs y]\ltimes
\C[S_N^y])$.

\medskip
Let
\be
U_{\bs 0}=\{g\in U(\gln[t])[\bs \la]\ |\
\ ge_{ii}=e_{ii}g,\ i=1\lc N\}
\ee
be the subalgebra of $U(\gln[t])[\bs \la]$ of $\gln$-weight
$\bs 0=(0\lc 0)$. Note that $\mc B_N\subset U_{\bs 0}$.

If $M$ is a $\gln[t]$-module, then
$U_{\bs 0}$ preserves the $\gln$-weight decomposition of $M$. In particular,
$U_{\bs 0}$ acts on $\V_{\bs 1}$.

Identify the space $\V_{\bs 1}$ with the space
$\C[\bs x]\otimes(\C[\bs y]\ltimes \C{S_N^y})$
by the linear isomorphism $\tilde \iota$:
\be
\tilde \iota:\ \V_{\bs 1}\to \C[\bs x]\otimes(\C[\bs y]\ltimes \C{S_N^y}),
\qquad p(\bs z)q(\bs \la) \ep_\tau
\mapsto q(\bs x)\otimes (\tau p(\bs y)),
\ee
for all $\tau\in S_N$ and all polynomials $p,q$.

\begin{lem}\label{second way}
The map $\tilde \iota$ identifies the image of the algebra $U_{\bs
0}$ in $\on{End} (\V_{\bs 1})$ with the algebra $\C[\bs
x]\otimes(\C[\bs y]\ltimes \C{S_N^y})$ acting by left
multiplications.
\end{lem}
\begin{proof}
The right multiplication by $x_i$ and $y_j$ correspond to
multiplications by $\la_i$ and $z_j$ in the space $\V_{\bs 1}$. The
right multiplication by $s_{ij}$ corresponds to switching the $i$-th
with the $j$-th factors and $z_i$ with $z_j$ in the space $\V_{\bs
1}$. Clearly, the algebra $U_{\bs 0}$ commutes with all these
operators. Therefore, the map $\tilde \iota$ identifies the image of
the algebra $U_{\bs 0}$ in $\on{End} (\V_{\bs 1})$ with a subalgebra
of left multiplications in $\C[\bs x]\otimes(\C[\bs y]\ltimes
\C{S_N^y})$.

Note that $\tilde \iota$ identifies the operators
$\la_i,e_{jj}\otimes t, e_{kl}e_{lk}\in U_{\bs 0}$ with the left
multiplications by $x_i, y_j, s_{kl}$, respectively. Since
$x_i,y_j,s_{kl}$ generate the $\C[\bs x]\otimes(\C[\bs y]\ltimes
\C{S_N^y})$, the lemma is proved.
\end{proof}
In particular, by Lemma \ref{second way} the image of the Bethe
subalgebra $\tilde\iota(\mc B_N)$ is a commutative subalgebra of
$\C[\bs x]\otimes(\C[\bs y]\ltimes \C{S_N^y})$. \medskip

Theorem \ref{Z=B} is equivalent to the following.
\begin{cor}
The subalgebras $\al(\mc Z_N)$ and $\tilde\iota(\mc B_N)$ of the
algebra $\C[\bs x]\otimes(\C[\bs y]\ltimes \C{S_N^y})$ coincide.
\qed
\end{cor}

\subsection{The spherical subalgebra}
Recall that $e\in H_N$ is the symmetrizing element, see \Ref{symm el}.
The {\it spherical subalgebra} $U_N$ is given by
\be
U_N=eH_Ne\subset H_N.
\ee
We have the Satake homomorphism:
\be
s:\ \mc Z_N\to H_N, \qquad c\mapsto c\ e.
\ee

Let $K$ be the $N\times N$ matrix with all entries $1$. Let
$X=\on{diag}(x_1\lc x_N)$, $Y=\on{diag}(y_1\lc y_N)$ be diagonal
$N\times N$ matrices.

Define the {\it universal spherical polynomial}
\be
\mc P^U=:\on{rdet}((v-X)(u-Y)-K):e.
\ee
We have
\be
s(\mc P^{\mc Z})=\mc P^U.
\ee

\begin{thm}
The coefficients of the universal spherical polynomial $c_{ij}e$
generate the spherical subalgebra $U_N$. In particular, the Satake
homomorphism is an isomorphism.
\end{thm}
\begin{proof}
The theorem follows from Theorem \ref{central thm}.
\end{proof}

The fact that Satake homomorphism is an isomorphism is not new, see \cite{EG}.

We use the isomorphism $\iota$ to identify the spherical subalgebra
with a subspace of $\V_{\bs 1}$. The left and right multiplications
by $\sigma\in S_N$ considered as elements of $H_N$ correspond to
two actions of the symmetric group $S_N$ on $\V$ which we call the
{\it left and right actions}. These actions are defined as follows.
The left action permutes the variables $\la_i$ and vectors $\ep_j$:
\be
\si^{L}\bigl(p(\bs z,\la_1\lc\la_N)\,
\ep_{i_1}\otimes\dots\otimes \ep_{i_N}\bigr)\,=\,
p(\bs z, \la_{\si(1)}\lc\la_{\si(N)})\,
\ep_{\si(i_1)}\otimes\dots\otimes \ep_{\sigma(i_N)}.
\ee
The left action restricted on $V^{\otimes N}$ coincides with
the standard Weyl group action on representations of $\gln$.

The right action permutes the variables $z_i$ and the factors of $V^N$:
\be
\si^{R}\bigl(p(z_1\lc z_N,\bs\la)\,\ep_{i_1}\otimes\dots\otimes
\ep_{i_N}\bigr)\,=\,
p(z_{\si(1)}\lc z_{\si(N)},\bs\la)\,
\ep_{i_{\si^{-1}(1)}}\otimes\dots\otimes \ep_{i_{\sigma^{-1}(N)}}.
\ee
Clearly the left and the right actions of $S_N$ commute.

The space $\V_{\bs 1}$ is invariant under the left and right $S_N$
actions. Denote by $\V_{\bs 1}^{S^L}$, $\V_{\bs 1}^{S^R}$ the
subspaces of invariants in $\V_{\bs 1}$ with respect to the left and
right actions respectively. Denote also by $\V_{\bs 1}^{S^{R}\times
S^L}=\V_{\bs 1}^{S^L}\cap\V_{\bs 1}^{S^R}$, the subspace of
invariants with respect to both actions.

\begin{lem}\label{He,eH,EHE}
For any $v\in \V_{\bs 1}$, $\si\in S_N$, we have
\be
\iota (\si^L v) = \si \iota v, \qquad \iota (\si^R v) = (\iota v) \si^{-1}.
\ee
Moreover
\be
\iota(\V_{\bs 1}^{S^R})=He\subset H,\qquad
\iota(\V_{\bs 1}^{S^L})=eH\subset H, \qquad
\iota(\V_{\bs 1}^{S^L\times S^R})=U_N\subset H.
\ee
\end{lem}
\begin{proof}
The lemma is straightforward.
\end{proof}

\begin{cor}
The Bethe algebra $\bar{\mc B}_N$ is isomorphic to the spherical
subalgebra $U_N\subset H_N$. Moreover, the space $\V_{\bs
1}^{S^L\times S^R}$ is a cyclic $\bar{\mc B}_N$-module which is
identified with the regular representation of $U_N$ by the isomorphism
$\iota$.
\qed
\end{cor}

\subsection{Action of the Cherednik algebra}
We identify the space $\V_{\bs 1}^{S^R}=\iota^{-1}(He)$ with the space
$\C[\bs z,\bs\la]$ and compute the left action of the Cherednik algebra.
Define the projection map:
\be
\on{pr}: \V_{\bs 1}\to \C[\bs z,\bs \la], \qquad \sum_{\si\in S_N}
p_\si(\bs z,\bs\la)\ep_\si\mapsto p_{id}(\bs z,\bs \la).
\ee

\begin{lem}\label{proj lem}
We have the isomorphisms of vector spaces:
\bea
&\on{pr}&:\ \V_{\bs 1}^{S^R} \to \C[\bs z,\bs \la], \\
&\on{pr}&:\ \V_{\bs 1}^{S^L\times S^R} \to P_N.
\eea
\end{lem}
\begin{proof}
The lemma is straightforward.
\end{proof}
In particular, we have
\be
(\on{pr}\iota^{-1}):\ H_N e \to\C[\bs z,\bs \la],\qquad
q(\bs x)p(\bs y)e\mapsto
q(\bs \la)p(\bs z).
\ee

Set
\bean\label{cher act}
\mc K_i=z_i+\sum_{j,\ j\neq i}s_{ij}^z\ \frac{1}{\la_i-\la_j}(1-s^\la_{ij}).
\eean

\begin{prop}
We have
\be
(\on{pr}\iota^{-1}) x_i= \la_i(\on{pr}\iota^{-1}), \qquad
(\on{pr}\iota^{-1}) s_{ij}= s_{ij}^{z,\la}(\on{pr}\iota^{-1}),\qquad
(\on{pr}\iota^{-1}) y_i=\mc K_i(\on{pr}\iota^{-1}).
\ee
In particular, the assignment $x_i=\la_i$, $y_i=\mc K_i$ and
$s_{ij}=s_{ij}^{z,\la}$ defines a left action of $H_N$ on $\C[\bs
z,\bs\la]$ such that $\C[\bs z,\bs\la]$ is canonically identified
with $H_Ne$ as a left $H_N$-module.
\end{prop}
\begin{proof}
The first and the second equalities are clear. To compute the action
of $y_1$ on $\C[\bs z,\bs\la]$ it is sufficient to compute it on
$\prod_{i=2}^N z_i^{a_i}$, since $\sum_{i=1}^N x_i$ and
$\sum_{i=1}^N y_i$ are central elements. This calculation is
straightforward.
\end{proof}

\begin{rem}
If $p(\bs y)$ is a symmetric polynomial in $y_1\lc y_N$, then it
is central, $p(\bs y)\in \mc Z_N$, and therefore it acts on $\C[\bs
z,\bs\la]$ as the operator of multiplication by $p(\bs z)$.
\end{rem}

\section{Calogero-Moser spaces} \label{sec C=B}
\subsection{The definition}
Let
\be \tilde C_N = \{(Z,\La)\in \gln\times \gln\ |\
\on{rk}([Z,\La]+1)=1\}.
\ee
The group $GL_N$ of complex invertible matrices acts on $\tilde C_N$
by simultaneous conjugation and the action is free and proper, see
\cite{Wn}. The quotient space $C_N$ is called the {\it $N$-th
Calogero-Moser space}. For $(Z,\La)\in\tilde C_N$, we write
$(Z,\La)\in C_N$ meaning the orbit of $GL_N$ containing $(Z,\La)$.

The Calogero-Moser space $C_N$ is a smooth complex affine variety of
dimension $2N$, see \cite{Wn}.
Let $\mc O^C_N$ be the {\it algebra of regular functions on $\C^N$}.
It is defined as follows. Let $z_{ij},\la_{ij}$ be the matrix entries
of $Z,\La$ considered as functions on $\tilde C_N$. Let
$\C[z_{ij},\la_{ij}]^{GL_N}$ be the algebra of polynomials in
$x_{ij},z_{ij}$ invariant with respect to the action of $GL_N$. Let
$I\subset \C[z_{ij},\la_{ij}]^{GL_N}$ be the ideal of the invariant
polynomials which vanish on $\tilde C_N$. Then
\be
\mc O^C_N=\C[z_{ij},\la_{ij}]^{GL_N}/I.
\ee

\subsection{The Calogero-Moser $\Psi$-function}
Let $Z=(z_{ij})$, $\La=(\la_{ij})$.
Define the {\it Calogero-Moser $\Psi$-function} by
\be
\Psi^{C}=\det(1-(v-\La)^{-1}(u-Z)^{-1}).
\ee
The Calogero-Moser $\Psi$-function (more precisely, the closely related
to it stationary Baker function) was introduced in \cite{Wn}.

The Calogero-Moser $\Psi$-function is a formal power series in $u^{-1}$
and $v^{-1}$ with coefficients in $\mc O^C_N$. Write
\be
\Psi^{C}=1+\sum_{i,j=1}^\infty \psi_{ij}^{C} u^{-i}v^{-j},
\qquad \psi_{ij}^{C}\in \mc O^C_N.
\ee

\begin{lem}\label{Pasha}
The algebra $\mc O^C_N$ of regular functions on $C^N$ is
generated by $\psi^C_{ij}$, $i,j\in\Z_{>0}$.
\end{lem}
\begin{proof}
For $\bs i=(i_1\lc i_k)$, $\bs j=(j_1\lc j_k)$, let
$T_{\bs i,\bs j}=\La^{i_1}Z^{j_1}\dots \La^{i_k}Z^{i_k}$.
According to the standard theory of invariants, the algebra
$\C[z_{ij},\la_{ij}]^{GL_N}$ is generated by functions
$\tr(T_{\bs i,\bs j})$, see \cite{W}.

Define a $\Z_{\geq 0}$ filtration on algebra $\mc O^C_N$ by letting
$|\bs i|=\sum_{s=1}^ki_s$, $|\bs j|=\sum_{s=1}^kj_s$ and $\deg \tr
T_{\bs i,\bs j}=|\bs i|+|\bs j|$. We say $\deg F\leq s$ if $F$ can be
written as a linear combination of products of $\tr T_{\bs i,\bs j}$
with degree of each factor at most $s$.

We claim that
\be
\tr T_{\bs i,\bs j}=\tr T_{|\bs i|,|\bs j|}+ ...,
\ee
where the dots denote the terms of degree less than $|\bs i|+|\bs j|$.

Given that claim, the proof of the lemma is similar to the
proof of Lemma \ref{multisym gen}.

To prove the claim, let $K=1-[\La,Z]$. Since the
rank of $K$ is one and $\tr K=N$ ,
we see that
\bea
\tr (T_{\bs i,\bs j}K T_{\bs i',\bs j'}K)=\frac{1}{N^2}
\tr (KT_{\bs i,\bs j}K KT_{\bs i',\bs j'}K)= \\
\frac{1}{N^2}
\tr (KT_{\bs i,\bs j}K)\tr(KT_{\bs i',\bs j'}K)=
\tr (T_{\bs i,\bs j}K)\tr(T_{\bs i',\bs j'}K).
\eea
It follows that
\be
\tr (T_{\bs i,\bs j}[\La,Z])=\tr(\La^{|\bs i|}Z^{|\bs j|}[\La,Z])+\dots,
\ee
where the dots denote the terms of degree
less than $|\bs i|+|\bs j|+2$.

Therefore it is sufficient to prove that for
$i,j\in\Z_{\geq 0}$, we have
\be
\tr (\La^iZ^j \La Z) =\tr( \La^{i+1}Z^{j+1})+\dots,
\ee
where the dots denote the terms of degree
less than $i+j+2$.

By the cyclic property of the trace we have
\be
\tr( \La^iZ^j \La Z)- \tr( \La^i\La Z^j Z) =-\tr (\La^iZ^j K)+\dots,
\ee
where the dots denote the terms of degree
less than $i+j+2$. On the other hand, commuting $\La$ through $Z^j$, we obtain
\be
\tr (\La^iZ^j \La Z)- \tr (\La^i\La Z^j Z)=\sum_{k=0}^{j-1}
\tr (\La^iZ^k K Z^{j-k})+\dots=j\tr (\La^iZ^j K)+\dots,
\ee
where the dots denote the terms of degree
less than $i+j+2$. The claim follows.
\end{proof}

Define the {\it Calogero-Moser universal polynomial} by
\be
\mc P^C=\det((v-\La)(u-Z)-1).
\ee
Write
\be
\Psi^{C}=\sum_{i,j=0}^N m_{ij}\ u^{N-i}v^{N-j}, \qquad m_{ij}\in \mc O^C_N.
\ee
\begin{lem}\
The algebra $\mc O^C_N$ of regular functions on $C^N$ is generated by
$m_{ij}$, $i,j=0,1\lc N$.
\end{lem}
\begin{proof}
We have
\be
\det(u-Z)=u^N+\sum_{i=1}^N m_{i0} u^{N-i},\qquad
\det(v-\La)=v^N+\sum_{j=1}^N m_{0j} v^{N-j}.
\ee
In particular, the coefficients of $\det(u-Z)^{-1}$ and $\det(v-\La)^{-1}$
are in the algebra generated by $m_{ij}$. Since
\be
\det(u-Z)^{-1}\det(v-\La)^{-1}\ \mc P^C=\Psi^C,
\ee
the lemma follows from Lemma \ref{Pasha}.
\end{proof}

\subsection{The Bethe algebra and the algebra of functions on the
Calogero-Moser space}
Recall that $\bar{\mc B}_N$ is the image of
the Bethe algebra $\mc B_N\in U(\gln[t])[\bs \la]$ in $\on{End}
(\V_{\bs 1})$ and $\bar b_{ij}$ are the generators of $\bar{\mc B}_N$.

Define the map
\be
\tau_{OB}: \ \mc O^C_N \to \bar{\mc B}_N,\qquad m_{ij}\mapsto \bar b_{ij}.
\ee

\begin{thm}\label{C=B}
The map $\tau$ is a well-defined algebra isomorphism.
\end{thm}
Theorem \ref{C=B} is proved in Section \ref{proof C=B}.

\begin{cor}\label{cor C=B} The space
$\V_{\bs 1}^{S^L\times S^R}$ is a cyclic $\bar{\mc B}_N$-module which is
isomorphic to the regular representation of $\mc O^C_N$.
\qed
\end{cor}

\section{Proof of Theorem \ref{C=B}}\label{proof C=B}

\subsection{Spaces of quasi-exponentials}\label{quasi sec}
For any $n\in\Z_{>0}$ and complex numbers \,$\la_1^0\lc\la_n^0$,
we call a complex vector space $W$ of dimension $n$ with a basis of the form
\,$q_1(u)e^{\la_1^0u}\lc q_n(u)e^{\la_n^0u}$, where $q_i(u)\in\C[u]$,
\,$i=1\lc n$, a {\it space of quasi-exponentials with exponents
$\bs\la^0=(\la_1^0\lc\la_n^0)$}.

Let
\be
\on{Wr}_W\,=\,c\,\det\left(\partial^{j-1}\bigl(q_i(u)
e^{\la_i^0 u}\bigr)\right)_{i,j=1}^n e^{-\sum_{i=1}^n\la_i^0u}
\ee
where $c\in\C$ is a non-zero constant such that $\on{Wr}_W$ is a monic
polynomial. The polynomial $\on{Wr}_W$ does not depend on the choice of a basis
in $W$ and is called the {\it Wronskian of W}. \,We denote the $\deg\on{Wr}_W$
simply by $\deg W$.

Zeros of Wronskian $\on{Wr}_W$ are called {\it singular points\/} of $W$.
The number of singular points counted with multiplicity equals $\deg W$.

We denote by $\mc D^{\mc W}_W$ the monic scalar differential operator of order
$n$ with kernel $W$. The operator $\mc D^{\mc W}_W$ is Fuchsian with singular
points exactly at the singular points of $W$ and infinity.

Write
\be
\mc D^{\mc W}_W=\partial^n+\sum_{i=1}^n G_{i,W}(u) \partial^{n-i}.
\ee
We have
\be
v^n+\sum_{i=1}^n G_{i,W}(u) v^{n-i}=
\frac{c}{\on{Wr}_W}\ \cdot\ \det\left(\begin{matrix}
q_1(u)e^{\la_1^0u} & \partial(q_1(u)e^{\la_1^0u})
&\dots & \partial^n(q_1(u)e^{\la_1^0u})\\
\dots & \dots & \dots &\dots \\
q_n(u)e^{\la_n^0u} & \partial(q_n(u)e^{\la_n^0u})
&\dots & \partial^n(q_n(u)e^{\la_n^0u}) \\
1 & v & \dots & v^n
\end{matrix}\right).
\ee

We call the function
\be
\Psi^{\mc W}_W\,=\,\Bigl(v^n+\sum_{i=1}^n G_{i,W}(u) v^{n-i}\Bigr)\,
\prod_{i=1}^n(v-\la_i^0)^{-1}
\ee
the {\it $\Psi$-function of the space W}. The $\Psi$-function is a
formal power series in $u^{-1}$ and $v^{-1}$
with complex coefficients. Moreover, it has the form
\be
\Psi^{\mc W}_W=1+\sum_{i,j=1}^\infty \psi_{ij,W}^{\mc W}u^{-i}v^{-j}, \qquad
\psi_{ij,W}^{\mc W}\in\C.
\ee

Let $W_1$ and $W_2$ be spaces of quasi-exponentials of possibly
different dimensions. We call the spaces $W_1$ and $W_2$ {\it
equivalent} if $\Psi^{\mc W}_{W_1}=\Psi^{\mc W}_{W_2}$. This defines
an equivalence relation on the set of spaces of quasi-exponentials.

We call a space of quasi-exponentials $W$ {\it minimal} if $W$ does
not contain a function of the form $e^{\la^0u}$ with $\la^0\in\C$.

For $\la^0\in \C$, let $W(\la^0)\subset W$ be the subspace spanned by
all function in $W$ of the form $q(u)e^{\la^0u}$, $q(u)\in\C[u]$. Then
$W(\la^0)$ is also a space of quasi-exponentials. We call a space of
quasi-exponentials $W$ {\it canonical} if for every $\la^0\in \C$ we
have the equality
\be
\dim W(\la^0)=\deg W(\la^0).
\ee
Note that for canonical spaces of quasi-exponentials we have $\dim
W=\deg W$ and for minimal spaces we have $\dim W\leq\deg W$.

\begin{lem}\label{equiv lem}
Each equivalence class of spaces of quasi-exponentials contains exactly
one minimal and exactly one canonical space of quasi-exponentials.
If $W_1$ and $W_2$ are equivalent spaces of quasi-exponentials, then
$\on{Wr}_{W_1}=\on{Wr}_{W_2}$ in particular, $\deg W_1=\deg W_2$.
If in addition, $W_1$ is minimal and $W_2\neq W_1$, then \,$\dim W_1<\dim W_2$.
\end{lem}
\begin{proof}
The lemma is straightforward.
\end{proof}
Define the {\it degree of an equivalence class of spaces of
quasi-exponentials} as the degree of any of representative of this
class. Denote the set of all equivalence classes of quasi-exponentials
of degree $N$ by $Q_N$.

We call a space of quasi-exponentials {\it generic} if all $\la_i^0$ are
distinct and all $q_i(u)$ are linear polynomials. A space of quasi-exponentials
is generic if and only if it is both canonical and minimal.

To a point $(\bs h^0,\bs \la^0) \in\C^N\times \C^N$ such that all $\la_i^0$
are distinct, we associate a generic space of quasi-exponentials
\be
W_{\bs h^0,\bs\la^0}=\langle (u-h_i^0)\ e^{\la_i^0 u}, \
i=1\lc N\rangle.
\ee
Generic spaces of quasi-exponentials of degree $N$
are in a bijective correspondence
with $S_N^{h,\la}$ orbits of points $(\bs h^0,\bs \la^0) \in\C^N\times
\C^N$ such that all $\la_i^0$ are distinct.

\subsection{The Bethe ansatz}
Let $\bs z^0=(z_1^0\lc z_N^0)$ and $\bs\la^0=(\la_1^0\lc
\la_N^0)$ be sequences of complex numbers.

Consider $V^{\otimes N}$ as the tensor product of evaluation
$\gln[t]$-modules with evaluation parameters $z_1^0\lc z_N^0$.
Namely, the action of $g\otimes t^k\in\gln[t]$ is given by
\be
(g\otimes t^k)(v_1\otimes\dots\otimes v_N)=
\sum_{i=1}^N (z_i^0)^k\ v_1\otimes\dots \otimes gv_i\otimes\dots \otimes v_N.
\ee
Let $\la_i$ act on $V^{\otimes N}$ as multiplication by $\la_i^0$:
\be
\la_i (v_1\otimes\dots\otimes v_N)=\la_i^0\ v_1\otimes\dots\otimes v_N.
\ee
Then $(V^{\otimes N})_{\bs 1}$ is a $\mc B_N$-module which we denote by
$(V^{\otimes N}(\bs z^0, \bs\la^0))_{\bs 1}$.

Let $v\in (V^{\otimes N}(\bs z^0, \bs\la^0))_{\bs 1}$ be
an eigenvector of the Bethe algebra,
$B_{ij} v= B_{ij,v} v$, where $B_{ij,v}\in\C$.
Consider the scalar differential operator
\be
\mc D^{\mc B}_v=\partial^N+
\sum_{i=1}^N\sum_{j=0}^\infty B_{ij,v}u^{-j}\partial^{N-i}.
\ee

\begin{lem}[\cite{quasi}]\label{distinct}
If $z_i^0\neq z_j^0$ and $\la_i^0\neq \la_j^0$ whenever $i\neq j$,
then the kernel of the operator $\mc D^{\mc B}_v$ is a generic space
of quasi-exponentials of degree $N$
with exponents $\bs \la^0$ and singular points
$\bs z^0$.
\qed
\end{lem}

We need a statement which says that generically the $\mc B_N$-module
$(V^{\otimes N}(\bs z^0, \bs\la^0))_{\bs 1}$ is the sum of
non-isomorphic one-dimensional modules which are in a bijective
corre\-spon\-dence with generic spaces of quasi-exponentials of degree
$N$ with exponents $\bs \la^0$ and singular points $\bs z^0$. Such a
statement is proved by the Bethe ansatz method.

Recall that the generic spaces of quasi-exponentials $W_{\bs h^0,\bs
\la^0}$ of degree $N$
are param\-e\-ter\-ized by $(\bs h^0,\bs\la^0)\in\C^N\times\C^N$
such that $\la_i^0$ are all distinct, see Section \ref{quasi sec}.
\begin{lem}\label{generic lem}
There exist Zariski open $S_N$ invariant subsets $\Theta$ and $\Xi$
of $\C^N\times\C^N$ such that
\begin{enumerate}
\item[(i)] For any $(\bs z^0,\bs \la^0)\in \Theta$, we have $z_i^0\neq
z_j^0$ and $\la_i^0\neq \la_j^0$ for $i\neq j$, and there exists a
basis for $(V^{\otimes N}(\bs z^0, \bs\la^0))_{\bs 1}$ such that
every basis vector $v$ is an eigenvector of the Bethe algebra, and
$\mc D^{\mc B}_v=\mc D^{\mc W}_{W_{\bs h^0,\bs \la^0}}$, where $(\bs
h^0,\bs \la^0)\in\Xi$.

\item[(ii)] For any $(\bs h^0,\bs \la^0) \in \Xi$, there exists a
unique up to a permutation $(\bs z^0,\bs\la^0)\in\Theta$ and a
unique up to proportionality vector $v\in (V^{\otimes N}(\bs z^0,
\bs\la^0))_{\bs 1}$ such that $v$ is an eigenvector of the Bethe
algebra and $\mc D^{\mc B}_v= \mc D^{\mc W}_{W_{\bs h^0,\bs
\la^0}}$.
\end{enumerate}
\end{lem}
\begin{proof}
This lemma is a special case of Lemma 6.1 in \cite{exp}.
\end{proof}

\subsection{The modules $\V_{\bs 1}(\bs a^0,\bs b^0)$}
Let $\si_i(\bs z)$ and $\si_i(\bs \la)$, $i=1\lc N$,
be the elementary symmetric functions
\be
\prod_{i=1}^N(u-z_i)=u^N+\sum_{i=1}^{N}(-1)^i\si_i(\bs z)u^{N-i}, \qquad
\prod_{i=1}^N(v-\la_i)=v^N+\sum_{i=1}^{N}(-1)^i\si_i(\bs \la)v^{N-i}.
\ee

For $(\bs a^0,\bs b^0)\in\C^N\times \C^N$ let $I_{\bs a^0,\bs
b^0}\subset\C[\bs z,\bs \la]$ be the ideal generated by the
functions $\si_i(\bs z)=a_i^0$ and $\si_i(\bs \la)=b_i^0$, $i=1\lc
N$.

Set
\be
\V_{\bs 1}(\bs a^0,\bs b^0)=\V_{\bs 1}^{S^L\times S^R}/
(\V_{\bs 1}^{S^L\times S^R}\bigcap
(V^{\otimes N})_{\bs 1}\otimes I_{\bs a^0,\bs b^0}).
\ee
The action of the Bethe algebra $\mc B_N$ in $\V_{\bs 1}$ induces an action
of the algebra $\mc B_N$ in the space $\V_{\bs 1}(\bs a^0,\bs b^0)$.

Let $(\bs z^0,\bs \la^0)\in\C^N\otimes \C^N$ be such that
$\si_i(\bs z^0)=a_i^0$, $\si_i(\bs \la^0)=b_i^0$:
\be
\prod_{i=1}^N(u-z_i^0)=u^N+\sum_{i=1}^{N}(-1)^ia_i^0u^{N-i}, \qquad
\prod_{i=1}^N(v-\la_i^0)=v^N+\sum_{i=1}^{N}(-1)^ib_i^0v^{N-i}.
\ee

The following lemma is proved by standard methods.
\begin{lem} We have
$\dim \V_{\bs 1}(\bs a^0,\bs b^0)=N!$. If all $z_i^0$ are distinct and all
$\la_i^0$ are distinct then the $\mc B_N$-modules $\V_{\bs 1}(\bs
a^0,\bs b^0)$ and $(V^{\otimes N}(\bs z^0,\bs \la^0))_{\bs 1}$ are
isomorphic.
\qed
\end{lem}

\begin{rem}
The $\mc B_N$-modules $\V_{\bs 1}(\bs a^0,\bs b^0)$ and $(V^{\otimes
N}(\bs z^0,\bs \la^0))_{\bs 1}$ are not always isomorphic. For
example, the former is always cyclic and the latter is not. If all
$\la_i^0$ are distinct then the module $\V_{\bs 1}(\bs a^0,\bs b^0)$
is isomorphic to a subspace in the Weyl module, see
\cite{exp}. It is interesting to understand the module
$\V_{\bs 1}(\bs a^0,\bs b^0)$ and it's precise relation to the
module $(V^{\otimes N}(\bs z^0,\bs \la^0))_{\bs 1}$ for all values
of parameters.
\end{rem}

\subsection{A relation of Calogero-Moser spaces
to the spaces of quasi-exponentials}
Let $(Z,\La)\in C_N$. Let the values of $\psi_{ij}^C$
on $(Z,\La)$ be $\psi_{ij,Z,\La}^{C}\in\C$. Then we
obtain a formal power series with complex coefficients:
\be
\Psi^C_{Z,\La}=1+\sum_{i,j=1}^\infty \psi_{ij,Z,\La}^{C} u^{-i}v^{-j}.
\ee

\begin{thm}[\cite{Wn}]\label{C=Q}
For any $(Z,\La)\in C_N$, there exists a space of quasi-exponentials
$W$ of degree $N$
such that the exponents of $W$ are eigenvalues of $Z$, the
singular points of $W$ are eigenvalues of $\La$ and
$\Psi^C_{Z,\La}=\Psi^{\mc W}_W$.
Moreover, this establishes a bijective correspondence between points
of $C_N$ and the set $Q_N$ of equivalence classes of spaces of
quasi-exponentials of degree $N$.
\qed
\end{thm}

We call $(Z,\La)\in C_N$ a {\it generic point} if $Z$ has a simple
spectrum. The set of generic points is dense in $C_N$, see \cite{Wn}.
The generic points correspond to equivalence classes of
quasi-exponentials of degree $N$ which contain a generic space of
quasi-exponentials.

\subsection{Proof of Theorem \ref{C=B}}
The proof is similar to the proof of Theorem 5.3 in \cite{transversal}.

First we show that the map $\tau_{OB}$ is well defined. Let a
polynomial $R(m_{ij})$ in generators $m_{ij}$ be equal to zero in $\mc
O^C_N$. We need to prove that $R(\bar b_{ij})$ is equal to zero in the
algebra $\bar{\mc B}_N$. Consider $R(\bar b_{ij})$ as a polynomial in
$z_1\lc z_N$ and $\la_1\lc\la_N$ with values in
$\on{End}\left((V^{\otimes N})_{\bs 1}\right)$. Let $\Theta$ be as in
Lemma \ref{generic lem}, and $(\bs z^0,\bs \la^0)\in\Theta$. Then by
part (i) of Lemma \ref{generic lem}, the value of the polynomial
$R(\bar b_{ij})$ at $z_1=z_1^0\lc z_N=z_N^0$ and
$\la_1=\la_1^0\lc\la_N=\la_N^0$ equals zero. Hence, the polynomial
$R(\bar b_{ij})$ equals zero identically.

Next we show that the map $\tau_{OB}$ is injective. Let a polynomial
$R(m_{ij})$ in generators $m_{ij}$ be a nonzero element of $\mc
O^C_N$. Then the value of $R(m_{ij})$ at a generic point $(Z,\La)\in
C_N$ is not equal to zero. Then by part (ii) of Lemma \ref{generic
lem}, the polynomial $R(\bar b_{ij})$ is not identically equal to
zero.

Finally, the map $\tau_{OB}$ is surjective since the elements $\bar
b_{ij}$ generate the algebra $\bar {\mc B}_N\,$.
\qed

\begin{rem}
Let $L_{\bs \nu^{i}}$, $i=1\lc k$, be irreducible finite-dimensional
$\gln$-modules corresponding to partitions $\bs\nu^{i}$.
The Bethe algebra $\mc B_N$ acts on the space $(\otimes_{i=1}^k
L_{\bs \nu^{i} }) \otimes \C[z_1\lc z_k,\la_1\lc\la_N]$. Let
$\tilde B_{ij}$ be the linear operators corresponding to the
operators $B_{ij}$ and let $\tilde {\mc B}_N$ be the algebra
generated by $\tilde B_{ij}$.

Let
\be
\Psi^{\tilde{\mc B}}\,=\,\Bigl(v^N+\sum_{i=1}^N
\sum_{j=0}^\infty \tilde
B_{ij}u^{-j}v^{N-i}\Bigr)\,\prod_{i=1}^N(v-\la_i)^{-1}.
\ee
Set
$n=\sum_{i=1}^k |\bs \nu^{i}|$. Then we have a map $\mc O^C_n\to
\tilde {\mc B}_N$, which sends the coefficients of $\Psi^C$ to the
corresponding coefficients of $\Psi^{\tilde{\mc B}}$. Similarly to
Theorem \ref{C=B}, using the results of \cite{transversal}, one can
show that this map is a well defined homomorphism of algebras.
However, it is neither injective nor surjective in general,
see Section 5.2 in \cite{Gaudin Ham}.
\end{rem}

\section{Corollaries of Theorems \ref{Z=B}, \ref{C=B}}\label{sec cor 2}
\subsection{Regular functions on the Calogero-Moser space and the
center of the Cherednik algebra}
Define an algebra homomorphism
\be
\tau_{OZ}: \mc O^C_N \to \mc Z_N, \qquad m_{ij}\mapsto c_{ij}.
\ee
\begin{cor}\label{O=C}
The map $\tau_{CZ}$ is a well-defined algebra isomorphism
and $\tau_{OZ}=\tau_{BZ}\circ\tau_{OB}$.
\end{cor}
\begin{proof}
By Theorems~\ref{C=B} and~\ref{Z=B}, the maps $\tau_{BZ}$ and $\tau_{OB}$
are algebra isomorphisms such that $\tau_{OB}(m_{ij})=\bar b_{ij}$ and
$\tau_{BZ}(\bar b_{ij})=c_{ij}$. The claim follows.
\end{proof}

The fact that algebras $\mc O^C_N$ and $\mc Z_N$ are isomorphic
is proved by a different method in \cite{EG}.

\subsection{Bijections}\label{sec bij}
Recall that we have the following sets.

\begin{itemize}

\item The Calogero-Moser space $C_N$.

\item The set $Q_N$ of equivalence classes of spaces of
quasi-exponentials of degree $N$.

\item The set $R_N$ of isomorphisms classes of the irreducible
representations of the Cherednik algebra $U_N$.

\end{itemize}
There are well known bijections between these three sets. The
bijection between $C_N$ and $Q_N$ is contained in \cite{Wn}, see also
Theorem \ref{C=Q}. The bijection between $C_N$ and $R_N$ is described
in \cite{EG}. We add one more set to this list:

\begin{itemize}
\item The set of eigenvectors of the Bethe algebra $\mc B_N$
up to a multiplication by a non-zero number in
\be
\bs\V_{\bs 1}=\bigoplus_{\bs
(\bs a^0,\bs b^0)\in\C^N\times \C^N}\mc V_{\bs 1}(\bs a^0,\bs b^0).
\ee
\end{itemize}
We denote this set by $E_N$.

We describe the bijections of $E_N$ to the first three sets. Let
$v\in\mc V_{\bs 1}(\bs a^0,\bs b^0)\subset \bs\V_{\bs 1}$ be an
eigenvector of the Bethe algebra $\mc B_N$. Let $B_{ij,v}\in\C$ be the
corresponding eigenvalues: $B_{ij}v=B_{ij,v}v$.

Note that the action of the algebra $\mc B_N$ in $\mc V_{\bs 1}(\bs
a^0,\bs b^0)$ factors through the action of the algebra $\bar{\mc
B}_N$. In particular, by Theorem \ref{C=B}, the algebra $\mc O^C_N$
acts on $\mc V_{\bs 1}(\bs a^0,\bs b^0)$. Moreover, an eigenvector
$v$ of the Bethe algebra $\mc B_N$ defines an algebra homomorphism $\chi_v: \mc
O^C_N\to\C$.

\begin{cor}
If $v,w\in \bs\V_{\bs 1}$ are eigenvectors of the Bethe algebra $\mc
B_N$ and $B_{ij,w}=B_{ij,v}$ for all $i,j$ then $w=cv$ for some
$c\in\C$.
\end{cor}
\begin{proof}
By Corollary \ref{cor C=B}, the space $V_{\bs 1}^{S^R\times S^L}$ is
a regular representation of the algebra $\mc O_C^N$. The regular and coregular
representations of the algebra $\mc O_C^N$ are isomorphic. Therefore 
the kernel of the ideal $\on{Ker}\chi_v$ is one-dimensional.
\end{proof}

Let $\nu_{C}: E_N\to C_N$ be the map which sends $v$ to the point in
$C_N$ corresponding to the maximal ideal $\on{Ker} \chi_v\subset \mc
O^C_N$.

\begin{cor}\label{C cor}
The map $\nu_C$ is a bijection.
\end{cor}
\begin{proof}
The corollary follows from Theorem \ref{C=B} and Corollary \ref{cor C=B}.
\end{proof}

Let $\nu_{Q}: E_N\to Q_N$ be the map which sends $v$ to the kernel
$W_v$ of the differential operator
\be
\mc D_v^{\mc B}=
\partial^N+\sum_{i=1}^N\sum_{j=0}^\infty B_{ij,v}u^{-j}v^{N-i}.
\ee

\begin{cor}\label{Q cor}
For every eigenvector $v\in\bs \V_{\bs 1}$ of the algebra $\mc B_N$,
the $W_v$ is a canonical space of quasi-exponential of degree $N$.
The map $\nu_{Q}$ is a bijection.
\end{cor}
\begin{proof}
The space $W_v$ is a space of quasi-exponentials of degree $N$ by
Corollary \ref{C cor} and Theorem \ref{C=Q}. This space is generic
for generic values of $\bs a^0,\bs b^0$, see Lemma \ref{distinct}.
It follows by continuity that $W_v$ is a canonical space of
quasi-exponentials of degree $N$. Therefore the map $\nu_Q$ is well
defined. The map $\nu_Q$ a bijection by Corollary \ref{C cor} and
Theorem \ref{C=Q}.
\end{proof}

Let
\be
{\tilde \V}_{\bs 1}(\bs a^0,\bs b^0)=\V_{\bs 1}^{S^R}/
(\V_{\bs 1}^{S^R}\bigcap
(V^{\otimes N})_{\bs 1}\otimes I_{\bs a^0,\bs b^0}).
\ee
and
\be
{\tilde {\bs\V}}_{\bs 1}=\bigoplus_{\bs
(\bs a^0,\bs b^0)\in\C^N\times \C^N}{\tilde {\mc V}}_{\bs 1}
(\bs a^0,\bs b^0).
\ee
Clearly, we have an inclusion
${\bs \V}_{\bs 1} \subset\tilde{\bs \V}_{\bs 1}$.

The space ${\tilde{\bs \V}}_{\bs 1}(\bs a^0,\bs
b^0)$ is the left $H_N$-module.
In particular, an eigenvector of the Bethe algebra
$v\in{\bs \V}_{\bs 1} \subset\tilde{\bs \V}_{\bs 1}$ defines an
algebra homomorphism $\chi_v: \mc Z_N\to\C$.

Let $\nu_{R}: E_N\to R_N$ be the map which sends $v$ to the
$H_N$-submodule $M_v$ of $\tilde{\bs \V}_{\bs 1}(\bs a^0,\bs b^0)$ generated by
$v$.

\begin{cor}\label{R cor}
For every eigenvector $v\in\bs \V_{\bs 1}$ of algebra $\mc B_N$,
$M_v$ is an irreducible representation corresponding to the central
character $\chi_v:\mc Z_N\to\C$. The map $\nu_R$ is a bijection.
\end{cor}
\begin{proof}
By Theorem \ref{irreps} the irreducible representations of $H_N$ are
determined by the central characters $\chi :\mc Z_N\to \C$ and have
the form $H_Ne\otimes_{\mc Z_N}\chi$. Recall that $\V_{\bs
1}^{S^R}$ is identified with $H_Ne$, see Lemma \ref{He,eH,EHE}.
By Corollary \ref{O=C}, the central characters of $H_N$ are in a
bijective correspondence with the points of the Calogero-Moser space
$C_N$ and by Corollary \ref{C cor} the points of the Calogero-Moser
space are in a bijective correspondence with the set $E_N$. The
corollary follows.
\end{proof}

\subsection{Example $N=2$}
The algebra $\bar B_2\simeq {\mc O}_2\simeq {\mc Z}_2$ can
be described by generators and relations as follows
\be
\C[g_1,g_2,h_1,h_2,T]/(T^2-h_1g_1T+(g_1^2-2g_2)h_2+(h_1^2-2h_2)g_2-1).
\ee
It is a free module of rank 2 over
the subalgebra $\C[g_1,g_2,h_1,h_2]$ generated by $1$ and $T$.
We describe the corresponding universal polynomials and
generators for all three algebras.

\medskip

The universal central polynomial has the form
\bea
\mc P^{\mc Z}= \hspace{14cm}\\ (1-(v-x_1)(u-y_1)-(v-x_2)(u-y_2)+
(v-x_1)(v-x_2)(u-y_1)(u-y_2))-s_{12}=
\\
v^2u^2-(y_1+y_2)v^2u-(x_1+x_2)vu^2+y_1y_2v^2+x_1x_2u^2+
((x_1+x_2)(y_1+y_2)-2)vu-
\\
((x_1+x_2)y_1y_2-(y_1+y_2))v-((x_1x_2(y_1+y_2)-(x_1+x_2))u+\\ 1+
x_1x_2y_1y_2-x_1y_1-x_2y_2-s_{12}.
\eea
In particular the generators $g_1,g_2,h_1,h_2,T$ of the
center $\mc Z_2$ of $H_2$ are given by
\be
x_1+x_2,\quad y_1+y_2,\quad x_1x_2,\quad y_1y_2,\quad x_1y_1+x_2y_2-s_{12}.
\ee

\medskip

The Calogero-Moser universal polynomial has the form
\bea
\mc P^{C}=\det((v-\La)(u-Z)-1)= \hspace{8cm}\\
u^2v^2-\tr(Z)v^2u-\tr(\La) vu^2+
\det(Z)v^2+\det(\La)u^2+
(\tr(\La)\tr(Z)vu +\\(\det(\La)\tr(Z)-\tr(\La))v+
(\det(Z)\tr(\La)-\tr(Z))u+1+\det(\La Z)-\tr(\La Z).
\eea
The generators $g_1,g_2,h_1,h_2,T$ of the
algebra $\mc O_2$ of the regular functions on
$C_2$ are given by
\be
\tr(\La),\quad \tr(Z),\quad \det(\La),\quad \det(Z),\quad \tr(\La Z).
\ee
Note that $\det(\La)=((\tr(\La))^2-\tr(\La^2))/2$,
$\det(Z)=((\tr(Z))^2-\tr(Z^2))/2$.

\medskip

The universal differential operator of $B_2$ has the form:
\bea
\mc D^{\mc B}=\partial^2-(\la_1+\la_2+ e_{11}(u)+e_{22}(u))\partial +\\
(\la_1+e_{11}(u))(\la_2+e_{22}(u))-e_{21}(u)e_{12}(u)-(e_{22}(u))'.
\eea
The universal Bethe polynomial has the form
\bea
\mc P^{\bar{\mc B}}=(u-z_1)(u-z_2)v^2-
((\la_1+\la_2)(u-z_1)(u-z_2)+2u-z_1-z_2)v+\\
1+\la_1\la_2z_1z_2-
\left(\begin{matrix}
\la_1z_1+\la_2z_2 & -1\\
-1 & \la_1z_2+\la_2z_1
\end{matrix}\right).
\eea
Here we used the basis $\{\ep_{id}=\ep_1\otimes \ep_2, \ep_{s_{12}}=
\ep_2\otimes \ep_1\}$. The space $\V_{\bs 1}$ is a free $\C[\bs z,\bs
\la]$-module of rank $2$ with generators $\{\ep_{id}, \ep_{s_{12}}\}$,
the action of $\bar{\mc B}_2$ commutes with multiplication by elements
of $\C[\bs z,\bs \la]$.

The generators $g_1,g_2,h_1,h_2,T$ of the
image $\bar{\mc B}_2$ of the Bethe algebra ${\mc B}_2$ are given by
\be
\la_1+\la_2,\quad z_1+z_2,\quad \la_1\la_2, \quad z_1z_2, \quad
\left(\begin{matrix}
\la_1z_1+\la_2z_2 & -1\\
-1 & \la_1z_2+\la_2z_1
\end{matrix}\right).
\ee

\bigskip

\address{EM: {\it Department of Mathematical Sciences,
Indiana University - Purdue University,
Indianapolis, 402 North Blackford St, Indianapolis,
IN 46202-3216, USA}}

\medskip

\address{VT: {\it Department of Mathematical Sciences,
Indiana University - Purdue University,
Indianapolis, 402 North Blackford St, Indianapolis,
IN 46202-3216, USA \and St.\,Petersburg Branch of Steklov Mathematical
Institute Fontanka 27, St.\,Petersburg, 191023, Russia}}

\medskip

\address{AV: {\it Department of Mathematics, University of
North Carolina at Chapel Hill, Chapel Hill, NC
27599-3250, USA}}

\end{document}